\documentclass[11pt]{amsart}
\usepackage{a4wide}
  \usepackage{amsmath}
  \usepackage{amssymb}
  \usepackage{latexsym}
  \usepackage[pdftex]{graphicx}
  \usepackage{color}
  \usepackage{url}
  \usepackage[margin=.9in]{geometry}

\newtheorem{theorem}{Theorem}[section]
\newtheorem{lemma}[theorem]{Lemma}

\newtheorem{claim}[theorem]{Claim}

\newtheorem{corollary}[theorem]{Corollary}

\newtheorem{problem}[theorem]{Open Problem}

\newcommand{\AC}{\mathcal A \mathcal C}
\newcommand{\bAC}{\overline{\AC}}

\newcommand{\N}{{\mathbb N}}

\newcommand{\E}{\mathbb E}

\newcommand{\Prob}{\mathbb{P}}

\newcommand{\eps}{\varepsilon}
\newcommand{\G}{{\mathcal{G}}}
\newcommand{\B}{{\mathcal{B}}}

\newcommand{\p}{{\bar{p}}}

  \newenvironment{proofof}[1]{\vspace{1ex}\noindent{\bf Proof of #1:}}{\hspace*{\fill}
  $\qed$\vspace{1ex}}

\newcommand{\Dcal}[0]{\ensuremath{{\mathcal D}}}

\newcommand{\Pcal}[0]{\ensuremath{{\mathcal P}}}

\newcommand{\Xcal}[0]{\ensuremath{{\mathcal X}}}

\newcommand{\eR}[0]{\ensuremath{ \mathbb R}}

\newcommand{\eN}[0]{\ensuremath{ \mathbb N}}

\newcommand{\norm}[1]{\ensuremath{\|#1\|}}

\newcommand{\Pee}[0]{\ensuremath{{\mathbb P}}}
\newcommand{\Ee}[0]{\ensuremath{{\mathbb E}}}

\newcommand{\isd}[0]{\hspace{.2ex} \raisebox{-.1ex}{$=$} \hspace{-1.5ex}
\raisebox{1ex}{{$\scriptstyle d$}} \hspace{.8ex} }

\DeclareMathOperator{\diam}{diam}

\DeclareMathOperator{\area}{area}
\DeclareMathOperator{\Po}{Po}
\DeclareMathOperator{\Bi}{Bi}
\DeclareMathOperator{\dd}{d}

\newcommand{\pr}[1]{{\bf(pr-#1)}}

\newcommand{\str}[1]{{\bf(str-#1)}}

\newcommand{\cruc}[1]{{\bf(cruc-#1)}}

\newcommand{\GPo}[0]{\ensuremath{G_{\Pcal}}}

\newcommand{\Rsde}[0]{\ensuremath{R_{\text{sde}}}}
\newcommand{\Rmdl}[0]{\ensuremath{R_{\text{mdl}}}}

\newcommand{\Rtil}[0]{\ensuremath{\tilde{R}}}

\begin{document}

\title{The acquaintance time of (percolated) random geometric graphs}

\author{Tobias M\"{u}ller}
\address{Mathematical Institute, Utrecht University, Utrecht, The Netherlands}
\email{\texttt{t.muller@uu.nl}}

\author{Pawe\l{} Pra\l{}at}
\address{Department of Mathematics, Ryerson University, Toronto, ON, Canada}
\email{\tt pralat@ryerson.ca}

\keywords{random graphs, random geometric graphs, vertex-pursuit games, acquaintance time}
\thanks{The authors gratefully acknowledge support from NSERC and Ryerson University}
\subjclass{05C80, 05C57, 68R10}

\begin{abstract}
In this paper, we study the acquaintance time $\AC(G)$ defined for a connected graph $G$. We focus on $\G(n,r,p)$, a random subgraph of a random geometric graph in which $n$ vertices are chosen uniformly at random and independently from $[0,1]^2$, and two vertices are adjacent with probability $p$ if the Euclidean distance between them is at most $r$. We present asymptotic results for the acquaintance time of $\G(n,r,p)$ for a wide range of $p=p(n)$ and $r=r(n)$. In particular, we show that with high probability $\AC(G) = \Theta(r^{-2})$ for $G \in \G(n,r,1)$, the ``ordinary" random geometric graph, provided that $\pi n r^2 - \ln n \to \infty$ (that is,  above the connectivity threshold). For the percolated random geometric graph $G \in \G(n,r,p)$, we show that with high probability $\AC( G ) = \Theta(r^{-2} p^{-1} \ln n)$, provided that $p n r^2 \geq n^{1/2+\eps}$ and $p < 1-\eps$ for some $\eps>0$. 
\end{abstract}

\maketitle

\section{Introduction and statement of results}\label{sec:intro}

In this paper, we study the following graph process, which was recently introduced by Benjamini, Shinkar, and Tsur~\cite{bst}. Let $G=(V,E)$ be a finite connected graph. We start the process by placing exactly one \emph{agent} on each vertex of $G$. Every pair of agents on adjacent vertices is declared to be \emph{acquainted}, and remains so throughout the process. In each round of the process, we choose some matching $M$ in $G$.  ($M$ need not be maximal; perhaps it is a single edge.) For each edge of $M$, we swap the agents occupying its endpoints, which may cause more agents to become acquainted. The \emph{acquaintance time} of $G$, denoted by $\AC(G)$, is the minimum number of rounds required for all agents to become acquainted with one another.

It is clear that 
\begin{equation}\label{eq:trivial_lower}
\AC(G) \ge \frac {{|V| \choose 2}}{|E|} - 1, 
\end{equation}
since $|E|$ pairs are acquainted initially, and at most $|E|$ new pairs become acquainted in each round. In~\cite{bst}, it was shown that always $\AC(G) = O(\frac{n^2}{\ln n / \ln \ln n})$, where $n = |V|$, which was slightly sharpened in~\cite{KMP} to $\AC(G) = O(\frac{n^2}{\ln n})$. This general upper bound was recently improved and now we know that $\AC(G) = O(n^{3/2})$ for every graph $G$, which was conjectured in~\cite{bst}
and is tight up to a multiplicative constant~\cite{AngelShinkar}. 
In~\cite{KMP}, another conjecture from~\cite{bst} on the acquaintance time of the random graph $\G(n,p)$ was proved. 
It was shown that asymptotically almost surely $\AC(G(n,p)) = O(\ln n / p)$, provided that $pn - \ln n - \ln \ln n \to \infty$ as $n \to \infty$ (that is, above the threshold for Hamiltonicity). Moreover, a matching lower bound for dense random graphs was provided, which also implies that asymptotically almost surely $K_n$ cannot be covered with $o(\ln n / p)$ copies of a random graph $\G(n,p)$, provided that $pn > n^{1/2+\eps}$ and $p < 1-\eps$ for some $\eps>0$. The problem is similar in flavour to the problems of Routing Permutations on Graphs via Matchings~\cite{ACG94}, Gossiping and Broadcasting~\cite{HHL88}, and Target Set Selection~\cite{KKT03, Che09, Rei12}.

\bigskip

In the present paper, we consider the acquaintance time of (percolated) random geometric graphs. If $V \subseteq \eR^2$ is a set of points and $r>0$ then the {\em geometric graph} $\G(V,r)$ is the graph with vertex set $V$ and an edge between two points if and only if their distance is at most $r$. Such a graph is also called a {\em unit disk graph} since it is the intersection graph of disks of the same radius (namely disks of radius $r/2$ centered on the points of $V$). Throughout this paper, we let $X_1,X_2,\ldots \in \eR^d$ be an infinite supply of random points, i.i.d.~on the unit square. For notational convenience (and following Penrose~\cite{PenroseBoek}) we set:
\begin{equation}\label{eq:Xdef}
\Xcal_n := \{ X_1,X_2, \dots, X_n \}.
\end{equation}

\noindent
The {\em random geometric graph} $\G(n, r)$ is the random graph obtained by taking $\Xcal_n$ as the vertex set, i.e.~$\G(n,r)  := \G(\Xcal_n, r)$. To prevent dealing with annoying trivial cases we shall always assume that $r < \sqrt{2}$ throughout this paper (otherwise $G \in \G(n,r)$ is a clique and $\AC(G) = 0$).

The study of random geometric graphs essentially goes back to Gilbert~\cite{Gilbert61} who defined a very similar model in 1961. For this reason it is often also called the {\em Gilbert model}. Random geometric graphs have been the subject of a considerable research effort in the last two decades. As a result, detailed information is now known on various aspects such as ($k$-)connectivity~\cite{PenroseMST1, Penrosekconn}, the largest component~\cite{PenroseBoek}, the chromatic number and clique number~\cite{twopoint, McDiarmidMuller11}, the (non-)existence of Hamilton cycles~\cite{BBKMW, MullerPerezWormald} and the simple random walk on the graph~\cite{CooperFriezeCoverRgg}. A good overview of the results prior to 2003 can be found in the monograph~\cite{PenroseBoek}.

The {\em percolated random geometric graph} $\G(n,r,p)$ is obtained by retaining each edge of $\G(n,r)$ with probability $p$ (and discarding it with probability $1-p$). To be more precise, for each edge of $\G(n,r)$ we flip a biased coin which is independent of $\Xcal_n$ and the other coin tosses for the other edges, and keep the edge if the coin comes up heads. In particular, $\G(n,r) = \G(n,r,1)$.
This model has not received the same amount of attention as the unpercolated random geometric graph, but very recently 
Penrose~\cite{PenroseArxiv} gave a very precise result on the threshold for connectivity.

\bigskip

As typical in random graph theory, we shall consider only asymptotic properties of $\G(n,r)$ and $\G(n,r,p)$ as $n\rightarrow \infty$, where both $r$ and $p$ may and usually do depend on $n$. 
Throughout this paper, we will say that a sequence of events $E_1, E_2, \dots$ holds {\em with high probability} (abbreviated w.h.p.) if $\Pee(E_n) \to 1$ as $n\to\infty$. 

\bigskip

It follows from a very precise result of Penrose~\cite{PenroseMST1} that the (classical) random geometric graph $\G(n,r_n)$ is w.h.p.~connected if and only if the sequence $(r_n)_n$ is such that $\pi n r_n^2 - \ln n \to \infty$ as $n\to\infty$. We are able to obtain the following result, which tells us the likely value of the acquaintance time up to a constant factor, whenever the acquaintance time is (w.h.p.) well defined.

\begin{theorem}\label{thm:gnr}
If $(r_n)_n$ is such that $\pi n r_n^2 - \ln n \to \infty$,  then $\AC( G(n,r_n) ) = \Theta(r_n^{-2} )$ w.h.p.
\end{theorem}

\bigskip
For the percolated random geometric graph $\G(n,r_n,p_n)$ we are slightly less successful.  
For dense graphs we determine the likely value of the acquaintance time up to a multiplicative constant, but the behaviour for sparser graphs remains undetermined.

\begin{theorem}\label{thm:gnrp}
Let $\eps > 0$ be arbitrary. If $(r_n)_n$ and $(p_n)_n$ are such that $p_n < 1-\eps$ and $p_n n r_n^2\geq n^{1/2+\eps}$, then $\AC(\G(n,r_n, p_n) ) = \Theta (r_n^{-2} p_n^{-1} \ln n)$ w.h.p.
\end{theorem}

In the course of the proof we will in fact prove slightly more. Namely, we will derive an upper bound of $\AC(G) = O( r_n^{-2} p_n^{-1} \ln n)$ that works whenever $p_n n r_n^2 \geq K \ln n$ for some large constant $K$.

%


\section{Preliminaries}

Throughout this paper $B(x,r) \subseteq \eR^2$ will denote the points at distance $<r$ to the point $x$. We will denote by $\Po(\lambda)$ the Poisson distribution with parameter $\lambda$, and $\Bi(n,p)$ will denote the binomial distribution with parameters $n$ and $p$. We will make use of the following incarnation of the Chernoff bounds. A proof can, for instance, be found in Chapter 1 of~\cite{PenroseBoek}.

\begin{lemma}\label{lem:chernoff}
Let $Z$ be either Poisson or Binomially distributed, and write
$\mu := \Ee Z$.
\begin{enumerate}
 \item For all $k \geq \mu$ we have
\[
\Pee( Z \geq k ) \leq e^{-\mu H(k/\mu)},
\]
\item For all $k \leq \mu$ we have
\[
\Pee( Z \leq k ) \leq e^{-\mu H(k/\mu)},
\]
\end{enumerate}
where $H(x) := x\ln x - x + 1$.
\end{lemma}

We will need the following standard and elementary result. A proof can be found in~\cite{BDFMS}.

\begin{lemma}\label{lem:span}
Let $G$ be a connected geometric graph. Then $G$ has a spanning tree of maximum degree at most five.
\end{lemma}


We will make use of some known results on the acquaintance time of graphs. The first one is a very recent result of Angel and Shinkar~\cite{AngelShinkar}. 

\begin{theorem}[\cite{AngelShinkar}]\label{thm:angel}
For every connected graph $G$ we have $\AC(G) \leq 20 \cdot \Delta(G) \cdot |V(G)|$.
\end{theorem}

Recall that $G[H]$ is the graph with vertex set 
$V(G) \times V(H)$ with an edge between $(u,v)$ and $(u',v')$ if 
either 1) $u=u'$ and $vv' \in E(H)$, or 2) $uu' \in E(G)$.
So, in particular, $G[K_s]$ is the graph we get by replacing each vertex of $G$ 
by an $s$-clique and adding all edges between the cliques corresponding to adjacent vertices of $G$.
We will use the following two straightforward observations.

\begin{lemma}\label{lem:boxtimes}
We have $\AC( G[K_s] ) \leq \AC( G )$ for
all connected $G$ and all $s \in \eN$.
\end{lemma}

\begin{proof}[Sketch of the Proof]
We partition the agents into groups of size $s$, corresponding to the $s$-cliques that have replaced the vertices of $G$. Each group is treated as a single agent, and the strategy that yields $\AC( G )$ is used.
\end{proof}

\begin{lemma}\label{lem:return}
Let $G$ be an arbitrary connected graph.
There is a strategy such that all agents get acquainted, and return to
their initial vertices in $2\cdot\AC(G)$ rounds.
\end{lemma}

\begin{proof}[Sketch of the Proof]
Follow the strategy that yields $\AC(G)$. Then, simply repeat the sequence of moves in reversed order.
\end{proof}

We will also use the fact, observed in~\cite{bst}, that for any graph $G$ on $n$ vertices with a Hamiltonian path, we have $\AC(G)=O(n)$. In fact, we need a slightly stronger statement that was proved in~\cite{KMP}. 

\begin{lemma}[\cite{KMP}]\label{lem:hamilton}
Let $G$ be a graph on $n$ vertices.  If $G$ has a Hamiltonian path, then there exists a strategy ensuring that within $2n$ rounds every pair of agents gets acquainted 
and, moreover, that every agent visits every vertex.
\end{lemma}

%

\section{Proof of the lower bound in Theorem~\ref{thm:gnr}}

The following Lemma is a simplification of Lemma A.1 in~\cite{twopoint}, where slightly more is proved.
 
\begin{lemma}[\cite{twopoint}]\label{lem:edges}
If the sequence $(r_n)_n$ is such that $n^2r_n^2 \to \infty$ then the number of edges of $\G(n,r_n)$ is
$\Theta( n^2r_n^2 )$ w.h.p.
\end{lemma}

The lower bound on $\AC$ now 
follows immediately from the trivial lower bound~(\ref{eq:trivial_lower}).
We see that when $\pi n r_n^2 = \ln n + \omega(1)$ then we have

\begin{equation}\label{eq:LBRGG}
\AC(G) \ge \frac{{n \choose 2}}{|E(G)|} - 1 = \Omega( r_n^{-2}  ) \quad \text{ w.h.p., }
\end{equation}

\noindent
and the proof of the lower bound is finished.

\section{The proof of the upper bound in Theorem~\ref{thm:gnr}}

It is convenient to split the proof into two cases.
The first, and easier, case is when the sequence $(r_n)_n$ is such that 
$\pi n r_n^2$ is at least $K \ln n$ for some large constant $K>0$.
We can then make use of the ``concentration phenomenon" to give a relatively easy proof of 
an upper bound of the right order of magnitude.
The second, and more involved, case is when $(r_n)_n$ is such that $\pi n r_n^2$ is somewhere between
$\ln n + \omega(1)$ and $K \ln n$.
Here, we need to make use the detailed information on the structure of $G(n,r_n)$ close
to the ``connectivity threshold". Luckily, a lot of this structural information has previously been 
obtained in, for instance,~\cite{BDFMS}, and we can obtain the statement we need for our proofs
by adapting some previous results to suit our needs.
We shall refer to the first case as the ``dense" case, and to the second as the ``sparse" case.

\subsection{The upper bound for $G(n,r_n)$ in the dense case} \label{sec:classic_dense}

For $(r_n)_n$, an arbitrary sequence of numbers with $0 < r_n < \sqrt{2}$, 
let us define $m_n := \lceil 1000/r_n \rceil$.
Then we have $1/m_n \leq r_n/1000$ and $1/m_n = \Omega(r_n)$.
(Moreover $1/m_n \sim r_n/1000$ if $r_n \to 0$.)
Let $\Dcal_n$ denote the dissection of the unit square into $m_n^2$ equal squares of dimensions
$(1/m_n)\times(1/m_n)$. We will call the squares of this dissections {\em cells}, and for a
given cell $c \in \Dcal_n$, we will denote by $V(c)$ the set of points 
of $\Xcal_n$ that fall in $c$.
Let $\mu_n := n / m_n^2$ denote the expectation $\Ee |V(c)|$.

\begin{lemma}\label{lem:wassenneus}
There is a constant $K > 0$ such that if $(r_n)_n$ is such that $\pi n r_n^2 \geq K \ln n$ then
$0.9 \cdot \mu_n \leq |V(c)| \leq 1.1 \cdot \mu_n$ for all $c \in \Dcal_n$, w.h.p.
\end{lemma}

\begin{proof} 
Fix a cell $c\in\Dcal_n$. By Lemma~\ref{lem:chernoff} the probability 
that $|V(c)| > 1.1 \cdot \mu_n$ satisfies

\[ \Pee( |V(c)| > 1.1 \cdot \mu_n ) \leq e^{ -\mu_n\cdot H(1.1) }, \]

\noindent
where  $H(x) = x\ln x-x+1$ is as in Lemma~\ref{lem:chernoff}.
Now notice that $m_n$ is non-decreasing in $r_n$, so that $\mu_n = n / m_n^2$ 
is non-increasing in $r_n$.
It follows that whenever $\pi n r_n^2 \geq K \ln n$ for some constant $K > 0$ then

\[ \Pee( |V(c)| > 1.1 \cdot \mu_n ) \leq \exp{\Big[} - (1+o(1)) \cdot  \frac {H(1.1) K}{\pi 10^6} \cdot \ln n {\Big ]}. \]

\noindent
(Here we have used that if $\pi n r_n^2 = K \ln n$ then $1/m_n \sim r_n/1000$.)
Hence, by the union bound,

\[ \begin{array}{rcl}
\Pee( \text{There exists a $c\in\Gamma_n$ with $|V(c)|> 1.1\cdot \mu_n$} )
& \leq & 
m_n^2 \cdot \exp{\Big[} - (1+o(1)) \cdot \frac {H(1.1)K}{\pi 10^6} \cdot \ln n {\Big ]}  \\
& \leq & 
n^2 \cdot n^{ - (1+o(1)) H(1.1) K \pi^{-1} 10^{-6}} \\
& = & o(1),  
\end{array} \]

\noindent
provided $K$ is chosen sufficiently large.
Completely analogously, we can show that, w.h.p., no cell will have
less than $0.9 \cdot \mu_n$ points, provided we chose $K$ sufficiently large.
\end{proof}

For the remainder of the proof, let $V \subseteq [0,1]^2, 0 < r < \sqrt{2}$ be such that 
the conclusion of this last lemma holds, but otherwise arbitrary.
It suffices to show that $G:=G(V,r)$ satisfies $\AC( G ) = O( r^{-2} )$.

For each $c\in\Dcal_n$ we partition $V(c)$ into three parts
$V_1(c), V_2(c), V_3(c)$, each of cardinality at most
$0.4 \cdot \mu_n$.
For each pair $1 \leq i < j \leq 3$ and each cell $c\in\Gamma_n$, let $W_{ij}(c) \subseteq V(c)$ be a set of cardinality exactly
$t := \lfloor 0.9 \mu_n\rfloor$ such that $V_i(c) \cup V_j(c) \subseteq W_{ij}(c)$; and set $W_{ij} := \bigcup_{c\in\Gamma_n} W_{ij}(c)$.
Let $G_{ij} = G[W_{ij}]$ denote the subgraph induced by $W_{ij}$.
We now observe that, since points in touching cells of the dissection $\Dcal_n$ have distance at most $r$, the graph $G_{ij}$ has 
a spanning subgraph that is isomorphic to $H[ K_t ]$ where $H$ denotes the
$m_n\times m_n$-grid.
It follows from Theorem~\ref{thm:angel} and Lemmas~\ref{lem:boxtimes} and~\ref{lem:return}
that we can acquaint all agents on vertices of $W_{ij}$ with each other, and
return them to their starting positions in $O( m_n^2 ) = O( r_n^{-2} )$ rounds.

By repeating this procedure for each of $W_{12}, W_{13}, W_{23}$, we acquaint all agents
with each other in $O(r_n^{-2})$ rounds, as required.

\subsection{Structural definitions and lemmas needed for the sparse case}

Before we can start the last part of the proof of Theorem~\ref{thm:gnr}, we need to recall some definitions and
(slightly adapted versions of) results from~\cite{BDFMS}.
Let us consider any geometric graph $G=(V,r)$, where  $V = \{ x_1, x_2, \ldots , x_n \} \subset [0,1]^2$. 

Let $m \in \eN$ be such that $s(m) := 1/m \leq r/1000$. Let $\Dcal = \Dcal(m)$ denote the \emph{dissection} of $[0,1]^2$ into
squares of side length $s(m)$.
We will call these squares {\em cells}.
Given $T > 0$ and $V \subseteq [0,1]^2$, we call a cell $c\in\Dcal$
{\em good} with respect to
$T, V$ if $|c\cap V| \geq T$ and {\em bad} otherwise.
When the choice of $T$ and $V$ is clear from the context we will just speak of good and bad.
Let $\Gamma = \Gamma(V, m, T, r)$ denote the  graph whose
vertices are the good cells of $\Dcal(m)$, with an
edge $cc' \in E(\Gamma)$ if and only if the lower left
corners of $c,c'$ have distance at most $r - s\sqrt{2}$.
(Note that this way, any $x\in c$ and $y\in c'$ have distance $\norm{x-y} \leq r$.)
We will usually just write $\Gamma$ when the choice of $V, m, T, r$ is clear from the
context.
Let us denote the components of $\Gamma$ by $\Gamma_1, \Gamma_2, \dots$ where
$\Gamma_i$ has at least as many cells as $\Gamma_{i+1}$ (ties are broken arbitrarily).
For convenience we will also write $\Gamma_{\max} = \Gamma_1$.
We will often be a bit sloppy and identify $\Gamma_i$ with the union of its cells, and
speak of $\diam(\Gamma_i)$ and the distance between $\Gamma_i$ and $\Gamma_j$ and so forth.


Let us call a point $v\in V$ {\em safe} if there is a good cell $c\in\Gamma_{\max}$ such
that $| B(v;r) \cap V \cap c | \geq T$. (That is, in the geometric graph $G(V;r)$, the point $v$ has at least $T$ neighbours inside $c$.)
Otherwise, if there is a good cell $c \in \Gamma_i$, $i \geq 2$, such that $| B(v;r) \cap V \cap c | \geq T$, we say that $v$ is {\em risky}.
Otherwise we call $v$ {\em dangerous}.

For $i \geq 2$ we let $\Gamma_i^+$ denote the set of all points of $V$ in cells of
$\Gamma_i$, together with all risky points $v$ that satisfy $| B(v;r) \cap V \cap c |
\geq T$
for at least one $c\in\Gamma_i$.
The following is a list of desirable properties that we would like $V$ and
$\Gamma(V,m,T,r)$ to have:

\begin{enumerate}
\item[\str{1}] $\Gamma_{\max}$ contains more than $0.99 \cdot |\Dcal|$ cells;
\item[\str{2}] $\diam(\Gamma_i^+)<r/100$ for all $i \geq 2$;
\item[\str{3}] If $u,v\in V$ are dangerous then
either $\norm{u-v} < r/100$ or $\norm{u-v} > r\cdot 10^{10}$;
\item[\str{4}] For all $i > j\geq 2$ the distance between $\Gamma_i^+$ and
$\Gamma_j^+$ is at least $r\cdot 10^{10}$;
\item[\str{5}] If $v\in V$ is dangerous and $i \geq 2$ then
the distance between $v$ and $\Gamma_i^+$ is at least $r \cdot 10^{10}$. 
\end{enumerate}

 Finally, we introduce some terminology for sets of dangerous and risky points.
Suppose that $V \subseteq [0,1]^2$ and $m, T, r$ are such that~\str{1}-\str{5} 
above hold.
Dangerous points come in groups of points of diameter
$< r/100$ that are far apart.
We formally define a {\em dangerous cluster} (with respect to $V,m,T,r$)
to be an inclusion-wise maximal subset of
$V$ with the property that $\diam(A) < r \cdot 10^{10}$ and all elements of
$A$ are dangerous.

A set $A \subseteq V$ is an {\em obstruction} (with respect to $V,m,T,r$)
if it is either a dangerous cluster or $\Gamma_i^+$ for some $i\geq 2$. We call $A$ an {\em $s$-obstruction} if $|A| = s$.
By \str{3}-\str{5}, obstructions are pairwise separated by distance $r \cdot 10^{10}$. 
(One consequence: a vertex in a good cell is adjacent in $G$ to at most one obstruction.)
A point $v\in V$ is {\em crucial} for $A$ if
\begin{enumerate}
\item[\cruc{1}] $A \subseteq N(v)$, and;
\item[\cruc{2}] $v$ is safe.
\end{enumerate}


We are  interested in the following choice of $m$ for our dissection.
For $n \in \eN$ and $\eta > 0$ a constant, let us define

\begin{equation}\label{eq:mdef}
 m_n := \left\lceil\sqrt{\frac{n}{\eta^2\ln n}}\right\rceil.
\end{equation}

The following lemma is almost identical to a lemma in~\cite{BDFMS}.
For completeness we spell out the adaptations that need to be made to its proof in Appendix~\ref{sec:structure}.

\begin{lemma}\label{lem:structure}
For every sufficiently small $\eta > 0$, there exists a $\delta = \delta(\eta) > 0$ such that 
the following holds.
Let $m_n$ be given by~\eqref{eq:mdef}, let $\Xcal_n$ be as in~\eqref{eq:Xdef},
let $T_n \leq \delta\ln n$ and let $r_n$ be such that $\pi n r_n^2 = \ln n + o(\ln n)$.
Then~\str{1}-\str{5} 
hold for $\Gamma(\Xcal_n,m_n,T_n,r_n)$ w.h.p.
\end{lemma}

The following lemma is also a slightly adapted version of a result in~\cite{BDFMS}. Its proof can be 
found in Appendix~\ref{sec:crucial}.

\begin{lemma}\label{lem:crucial}
For every sufficiently small $\eta > 0$, there exists a $\delta = \delta(\eta) > 0$ such that 
the following holds.
Let $(m_n)_n$ be given by~\eqref{eq:mdef}, let $T_n \leq \delta \ln n$ and let $V_n := \Xcal_n$
with $\Xcal_n$ as in~\eqref{eq:Xdef},
let $(r_n)_n$ be a sequence of positive numbers
such that $\pi r_n^2 - \ln n \to \infty$. \\
Then, w.h.p., it holds that for every $s \geq 2$, every $s$-obstruction has at least
$s - 100$ crucial vertices.
\end{lemma}

The proof of the following lemma is analogous to that of Lemma~\ref{lem:wassenneus} and is left to the reader.

\begin{lemma}\label{lem:maxcell}
If $\eta > 0$ is fixed, $(m_n)_n$ is as given by~\eqref{eq:mdef} and $V_n := \Xcal_n$
then there is a constant $C$ such that, w.h.p., every cell contains at most
$C \ln n$ points.
\end{lemma}
%

\subsection{The proof of the upper bound in Theorem~\ref{thm:gnr} in the sparse case}

It remains to prove the upper bound of Theorem~\ref{thm:gnr} for a sequence $r_n$ such that $\pi n r_n^2 = \ln n + \omega(1)$ and 
$\pi n r_n^2 \leq K \ln n$ for a large constant $K$.
For this range it suffices to consider the case when $1 \ll \pi n r_n^2 - \ln n \ll \ln n$ and prove that in that case 
w.h.p.~$\AC(\G(n,r_n)) = O( n / \ln n )$.
(By~\eqref{eq:LBRGG} we already have the asymptotically almost sure lower bound $\AC( \G(n,r_n) ) = \Omega( r_n^{-2} ) = \Omega( n / \ln n )$ for all
sequences $(r_n)_n$ satisfying $\pi n r_n^2 \leq K\ln n$.) 
Let us thus pick such a sequence $r_n$ and assume $\eta, \delta$ etc.~have been chosen in such a
way that the conclusions of Lemma~\ref{lem:structure} and~\ref{lem:crucial} hold w.h.p.
We also know from~\cite{PenroseMST1} that in this range $\G(n,r_n)$ is connected w.h.p.

In the sequel of the proof we let $V \subseteq [0,1]^2$ be an arbitrary set of points, and $r, m, \eta, \delta > 0$ 
be arbitrary numbers such that $G(V,r)$ is connected and the conclusions
of Lemma's~\ref{lem:structure}, \ref{lem:crucial} and~\ref{lem:maxcell} are satisfied. 
It suffices to show that every such graph $G = G(V,r)$ has acquaintance time $O( n / \ln n )$, as we will now show.

\begin{claim}\label{clm:aap}
There is a constant $c_1$ such that every obstruction consists of at most $c_1 \ln n$ points.
\end{claim}
\begin{proof}
Every obstruction $O$ has diameter $r/100$ and hence 
there are only $O(1)$ cells that contain points of $O$.
Since each cell contains $O(\ln n)$ points, we are done.
\end{proof}

Each point $v$ that is safe, but not in a cell of $\Gamma_{\max}$ has at least 
$T$ neighbours in some cell $c\in\Gamma_{\max}$.
We arbitrarily ``assign" $v$ to such a cell.

\begin{claim}\label{clm:noot}
There is a constant $c_2 > 1$ such that the following holds.
For every obstruction $O$ there is a cell $c \in \Gamma_{\max}$ such that 
at least $|O|/c_2$ vertices that are crucial for $O$ have been assigned to
$c$.
\end{claim}
\begin{proof}
Since $G$ is connected, every obstruction 
has at least one crucial vertex. (It is adjacent to at least one 
vertex $v \in V\setminus O$ and this $v$ cannot be dangerous or risky.)
This shows that, by choosing $c_2$ sufficiently large, the claim holds
whenever $|O| < 1000$.
Let us thus assume $|O| \geq 1000$.
Since $O$ has diameter $< r/100$, there is a constant $D$ 
such that the crucial vertices are all assigned to all one of the $D$
cells within range $2r$ of $O$.
Since $|O| > 1000$ there are at least $|O|-100 > |O|/2$ crucial vertices for $O$, and
hence at least $|O|/2D$ of these crucial vertices are assigned to the same cell.
\end{proof}

For each obstruction $O$, we now assign all its vertices to a
cell $c\in\Gamma_{\max}$ such that at least $|O|/c_2$ vertices that are crucial for $O$ have been 
assigned to $c$.
For a cell $c \in \Gamma$ let $V(c)$ denote the set of points that fell in $c$.
For each $c \in\Gamma_{\max}$, let $A(c)$ denote the union of $V(c)$ with 
all vertices that have been assigned to $c$.
Let us remark that:

\begin{claim}\label{clm:mies}
There exists a constant $c_3$ such that 
$|A(c)| \leq c_3 \ln n$ for all $c\in\Gamma_{\max}$.
\end{claim}
\begin{proof}
Note that if a vertex $v$ has been assigned to $c$, it must lie within
distance $2r$ of $c$.
Hence there are only $O(1)$ cells in which $A(c)$ is contained, and
since every cell contains $O(\ln n)$ points, we are done.
\end{proof}

Let us now partition $A(c)$ into sets $A_1(c), A_2(c), \dots, A_L(c)$ each of size
at most $T/100$, with $L = \lceil c_3/(100\delta)\rceil$.
The following observation will be key to our strategy.

\begin{claim}\label{clm:duif}
There is a constant $c_4$ such that the following holds.
For every $c \in \Gamma_{\max}$ and every $A' \subseteq A(c)$ with $|A'| \leq T/50$ there
is a  sequence of at most $c_4$ moves that results in the agents on vertices on $A'$ being
placed on vertices of $V(c)$ (and uses only edges of $G[A(c)]$).
\end{claim}
\begin{proof}
We first move all agents on vertices of $A' \setminus V(c)$ on safe vertices 
not in $V(c)$ onto vertices of $V(c) \setminus A'$ in one round.
(To see that this can be done, note that $|A'| < T/50$ and each safe vertex of $A'$ has
at least $T$ neighbours in $V(c)$.)
Let $W \subseteq V(c)$ be the set of vertices now occupied by agents that were originally on $A'$.

If $A'$ also contains (part of) some obstruction $O$, then we 
partition $O \cap A'$ into $O(1)$ sets $O_1, O_2, \dots, O_K$ of
cardinality at most $|O|/c_2$ where $K \leq \lceil 1/c_2\rceil$ (and hence is
a constant).
We first move the agents on vertices $O_1$ onto crucial vertices assigned to $c$, and
then on vertices of $V(c) \setminus W$, in two rounds.
(Note this is possible since $|A'| \leq T/50$ and each crucial vertex is
adjacent to at least $T$ vertices of $V(c)$.)
Similarly, supposing that the agents on $O_1, O_2, \dots, O_{i-1}$ have already been moved
onto vertices of $V(c)$, we can move the 
agents on vertices of $O_i$ onto vertices of $V(c) \setminus W$ not occupied by 
agents from $O_1, O_2, \dots, O_{i-1}$ in two rounds.

We thus have moved all agents on vertices of $A'$ onto vertices of $V(c)$ in constant many rounds, as required.
\end{proof}

We are now ready to describe the overall strategy.
Let us write $A_i := \bigcup_{c\in\Gamma_{\max}} A_i(c)$.
For each pair of indices $1 \leq i < j \leq L$ we do the following.

First we move all agents of $A_i(c) \cup A_j(c)$ onto vertices of
$V(c)$ (in constantly many moves, simultaneously for all cells $c\in\Gamma_{\max}$).
Next, we select a set $B(c) \subseteq V(c)$ for each $c\in\Gamma_{\max}$ with $|B(c)| = T$ and
all agents that were on $A_i(c) \cup A_j(c)$ originally are now on vertices of $B(c)$.
By Lemma~\ref{lem:span}, the largest component $\Gamma_{\max}$ of the cells-graph has
a spanning tree $H$ of maximum degree at most five.
Thus, by Theorem~\ref{thm:angel}, we have $\AC(\Gamma_{\max}) \leq O( \Delta(H) \cdot |V(H)| ) = O( |\Gamma| ) =  O( n/\ln n )$.
Now note that the graph spanned by $\bigcup_{c \in\Gamma_{\max}} B(c)$ contains a spanning
subgraph isomorphic to $\Gamma_{\max}[K_T]$.
Using Lemma~\ref{lem:boxtimes} and~\ref{lem:return}, we can thus acquaint all vertices of $\bigcup_{c\in\Gamma_{\max}} B(c)$
with each other and return them to their starting vertices in $O(n/\ln n)$ rounds.
So, in particular, we have acquainted the agents of $A_i \cup A_j$ and returned them
to their starting positions, in $O( n / \ln n )$ moves.

Once we have repeated this procedure for each of the ${L \choose 2} = O(1)$ pairs of indices all agents will
be acquainted, still in $O( n / \ln n )$ rounds. This concludes the (last part of the) proof of Theorem~\ref{thm:gnr}.

\section{The proof of the lower bound in Theorem~\ref{thm:gnrp}}

%
%
%

In hopes of doing better than the trivial lower bound~\eqref{eq:trivial_lower} on the acquaintance time of $\G(n,r_n,p)$, we consider a variant of the original process. This approach was used for binomial random graphs~\cite{KMP} and, after some adjustments combined with some additional averaging type argument, can be used here as well. Suppose that each agent has a helicopter and can, on each round, move to any vertex she wants. (We retain the requirement that no two agents can occupy a single vertex simultaneously.)  In other words, in every step of the process, the agents choose some permutation $\pi$ of the vertices, and the agent occupying vertex $v$ flies directly to vertex $\pi(v)$, regardless of whether there is an edge or even a path between $v$ and $\pi(v)$. (In fact, it is no longer necessary that the graph be connected.) Let the {\em helicopter acquaintance time} $\bAC(G)$ be the counterpart of $\AC(G)$ under this new model, that is, the minimum number of rounds required for all agents to become acquainted with one another. Since helicopters make it easier for agents to get acquainted, we immediately get that for every graph $G$, 
\begin{equation}\label{eq:bACvsAC}
\bAC(G) \le \AC(G).
\end{equation}
On the other hand, $\bAC(G)$ also represents the minimum number of copies of a graph $G$ needed to cover all edges of a complete graph of the same order. Thus inequality~(\ref{eq:trivial_lower}) can be strengthened to $\bAC(G) \ge {|V| \choose 2}/|E| - 1$.

In order to prove the lower bound in part (ii) of Theorem~\ref{thm:gnrp}, we prove the following general result.
If $G$ is a graph then we denote by $G^p$ the random subgraph of $G$ in which every edge is kept with probability $p$ 
and discarded with probability $1-p$ (independently of all other edges). So, in particular, $K_n^p$ is the familiar binomial
random graph $G(n,p)$ and $\G(n,r_n, p)$ is the same as $\G^p(n,r_n)$.

\begin{theorem}\label{thm:lower_gnp}
Let $\eps > 0$ be arbitrary. If $(G_n)_n$ is a sequence of graphs with $v(G_n)=n$, 
and $(p_n)_n$ is a sequence of edge-probabilities satisfying $p_n \le 1-\eps$, 
and $p_n e(G_n) \geq n^{3/2 + \eps}$ for all $n$, then
$$
\bAC(G_n^{p_n}) 
=  \Omega \left( \frac{n^2\ln n}{p_n \cdot e(G_n)} \right) \quad \text{ w.h.p. }
$$
\end{theorem}
\begin{proof}
Let $a_1, a_2, \ldots, a_n$ denote the $n$ agents, and let $A = \{a_1, a_2, \ldots, a_n\}$. Take 
$$
k = \frac {\eps}{20} \left(n^2/e(G_n)\right)\cdot \log_{1/(1-p_n)} n = \Theta\left(  \frac{n^2\ln n}{p_n \cdot e(G_n)} \right)
$$ 
and fix $k$ bijections $\pi_i : A \to V(G_n)$, for $i \in \{0, 1, \ldots k-1\}$.  
This corresponds to fixing a $(k-1)$-round strategy for the agents; in particular, agent $a_j$ occupies vertex $\pi_i(a_j)$ in round $i$. We aim to show that at the end of the process (that is, after $k-1$ rounds) the probability that all agents are acquainted is only $o((1/n!)^k)$. This will complete the proof. Indeed, the number of choices for $\pi_0, \pi_1, \ldots, \pi_{k-1}$ is $(n!)^k$, so by the union bound, w.h.p.\ no strategy makes all pairs of agents acquainted.

We say that a pair of vertices (or agents) is \emph{reachable} in a particular round if they are on vertices that are adjacent
in the non-percolated graph $G_n$.  (So a pair of agents can be reachable during one round but not reachable during another one.) 
Let $b$ be the number of pairs of agents that are reachable during more than $10 k n^{-2} e(G_n)$ rounds (with respect to the given agents' strategy, that is, the $k$ bijections that are fixed). 
In each round, at most $e(G_n)$ pairs are reachable.
It follows that $b \cdot 10 k n^{-2} e(G_n)\le e(G_n) k$, so that 

\[ b \le n^2 / 10 \le {n \choose 2}/2. \] 

\noindent
Hence, at least half of all pairs of agents are reachable during at most $10 k e(G_n) / n^2$ rounds. 
We call these pairs of agents \emph{important}.

To estimate the probability that a given agents' strategy makes all pairs of important agents acquainted, we consider the following analysis, which iteratively exposes edges of a percolated random graph $G_n^{p_n}$.  
For any pair $q = \{a_x, a_y\}$ of important agents, we consider all reachable pairs of vertices visited by this pair of agents throughout the process:
$$
S(q) = \{ e \in E(G_n) :  e = \pi_i(a_x)\pi_i(a_y) \text{ for some } i \in \{0, 1, \ldots k-1\} \}.
$$
Since $q$ is important, $1 \le |S(q)| \le 10 k e(G_n) / n^2$. 
Let us now fix an arbitrary ordering $q_1, q_2, \dots, q_m$ of our important pairs and consider the following process.
We take the first pair $q_1$ of important agents and expose the edges of $G_n^{p_n}$ in $S(q_1)$, one by one until
we either find an edge that is present in $G_n^{p_n}$ or we have exposed all of $S(q_1)$.  
If we expose all of $S(q_1)$ without discovering an edge, then the pair $q_1$ never gets acquainted and we halt our procedure. 
If instead we do discover some edge $e$ of $G_n^{p_n}$, then we discard all pairs of important agents that ever occupy this edge (that is, we discard all pairs $q$ such that $e \in S(q)$).  We now shift our attention to the next pair $q_i$ of important agents that we did not yet discard and repeat the procedure of exposing edges $S(q_i)$ until we find one that acquaints $q_i$ or we run out of edges. 
It may happen that some of the pairs of vertices in $S(q_i)$ have already been exposed, but the analysis guarantees that no edge has yet been discovered. 

We continue this process until either we have found an important pair that never gets acquainted or all available pairs of important agents have been investigated. Considering one pair of important agents can force us to discard at most $k$ important pairs (including the original pair) 
since in each round the edge acquaints at most one pair.
Hence, the process investigates at least $\frac 12 {n \choose 2} / k$ pairs of important agents. 
Moreover, writing $E_t := \{\text{$q_1, q_2, \dots,q_{t-1}$ get acquainted and $q_t$ did not get discarded}\}$, we have
\[ \begin{array}{rcl}
\Prob( \text{$q_t$ gets acquainted} | E_t ) 
& \le & 
1-(1-p)^{|S(q_t)|} \\
& \le & 
1-(1-p)^{10 k n^{-2} e(G_n)}.
\end{array} \]

\noindent
Hence, we find
\begin{eqnarray*}
\Prob(\text{all pairs acquainted}) 
& \leq & 
\Prob(\text{all important pairs acquainted} ) \\
&\le& \left( 1-(1-p)^{10 k n^{-2} e(G_n)} \right)^{\frac 12 {n \choose 2} / k} \\
& \le& \exp \left[ - (1-p)^{10 k n^{-2} e(G_n)} \cdot \frac 12 {n \choose 2} / k \right] \\
&\le& \exp \left[  -n^{-\eps/2} \cdot \frac12 {n \choose 2} /k \right] \\
& = & \exp \left[ - \Omega( n^{-\eps/2} p_n e(G_n) / \ln n )  \right], 
\end{eqnarray*}
using that $k = \frac {\eps}{20} \left(n^2/e(G_n)\right)\cdot \log_{1/(1-p_n)} n$ for the third line and that 
$k = \Theta( n^2\ln n / (p_n e(G_n))$ for the last line.
Now note that 
\begin{eqnarray*}
(n!)^k 
& \leq & n^{k \cdot n} \\
& = & \exp\left[ k \cdot n \ln n  \right] \\
& = & \exp\left[ O\left( n^3 \ln^2 n / (p_n e(G_n)) \right) \right].
\end{eqnarray*}

\noindent
Since $p_n e(G_n) \geq n^{1/2+\eps}$ we also have that
$n^{-\eps/2} p_n e(G_n) / \ln n \gg n^3 \ln^2 n / (p_n e(G_n))$, 
and hence 

\[ \Prob( \text{all pairs acquainted} ) = o( (1/n!)^k), \]

\noindent
which concludes the proof by a previous remark.
\end{proof}

\noindent
Combining the last theorem with Lemma~\ref{lem:edges}, we immediately get:

\begin{corollary}\label{cor:lower_gnp}
Let $\eps > 0$ be arbitrary. If the sequences $(r_n)_n$ and $(p_n)_n$ are such that $p_n < 1-\eps$ for all $n$ 
and $p_n^2 n r_n^2 \ge n^{1/2+\eps}$, then 
$$
\AC(\G(n,r_n,p_n)) \ge 
\bAC(\G(n,r_n,p_n)) = \Omega \left( r_n^{-2} p^{-1} \ln n \right) \quad \text{ w.h.p. }
$$
\end{corollary}


\section{The proof of the upper bound in Theorem~\ref{thm:gnrp}}

Let us start with the following useful observation. Let $p=p(n)$ and $t=t(n)$, and let $\B(t,p)$ be the ``standard" 
random bipartite graph with bipartite sets $X$ and $Y$ such that $|X|=|Y|=t$. 
For each pair of vertices $x \in X$ and $y \in Y$, we introduce an edge $xy$ with probability $p$, independently of all other edges. 
We will consider the probability that $\B(t,p)$ has a perfect matching. Very precise information is
already known about perfect matchings in this random graph model (see for instance~\cite{Bollobas2ndEd}, Section 7.3). 
We however need precise  quantitative bounds on the probability of existence of a perfect matching, which do not appear 
to exist in the literature as far as we are aware of.

\begin{lemma}\label{lem:perfect_matching}
With $\B(t,p)$ the random bipartite graph as above, we have
\[ \Pee( \B(t,p) \text{ has a perfect matching }) = 1 - O\left( t \cdot e^{ - \gamma t p} \right), \]

\noindent
for some universal constant $\gamma > 0$.
\end{lemma}

\begin{proof}
Let us first observe that if $tp \leq K \ln t$ for some constant $K$ then there is nothing to prove since we may assume, without loss on generality, that
$\gamma > 0$ is sufficiently small for $t \cdot e^{ - \gamma t p} \to \infty$ to hold in this case.
Let us thus assume that $tp > K \ln t$ in the sequel, where $K > 0$ is a constant to be chosen more precisely later on in the proof.

 Set $s_0 = \max\{ s \in \N: ps \le 1\}$. Let $S \subseteq X$ with $|S|=s \le s_0$. 
The number of vertices of $Y$ adjacent to at least one vertex from $S$ is the binomial random variable $Z \isd \Bi( t, 1-(1-p)^s )$, whose 
expected value is
$$
( 1-(1-p)^s )t  \ge (1-e^{-ps})t \ge (1-e^{-1}) pst.
$$
Hence, applying the Chernoff bound (Lemma~\ref{lem:chernoff}), the set $S$ fails the Hall condition with probability at most
$$
\Prob( Z < s ) = \Prob\left[Z < \frac{\E[Z]}{(1-e^{-1})pt} \right] \le 
\exp\left[ - \E Z \cdot H\left( \frac{1}{(1-e^{-1})pt} \right) \right]
\leq e^{- s tp / 100},
$$
where $H(x) = x\ln x -x + 1$ and we have used that $\lim_{x\downarrow 0} H(x) = 1$ and the last inequality holds
for $t$ sufficiently large.
Hence, the probability that the necessary condition in the statement of Hall's theorem fails for at least one set $S$ with $|S| \le s_0$ is at most
$$
\sum_{s=1}^{s_0} {t \choose s} e^{-stp/100} \le \sum_{s=1}^{s_0} t^s e^{-stp/100} = 
\sum_{s=1}^{s_0} \left( t e^{-tp/100} \right)^s = 
O( t e^{-tp/100} ),
$$
using that $t e^{-tp/100} = o(1)$.
Now, let $0 < \eps < (1-e^{-1})/2$ be a constant so that $\sum_{s \le \eps t} {t \choose s} \le \exp( 0.08 t)$. Consider any set $S \subseteq X$ with $s_0 < |S| = s \le \eps t$. The expected size of $N[S]$ is at least $(1-e^{-1})t$. It follows from Chernoff bound (see Lemma~\ref{lem:chernoff}) that 
$$
\Prob \Big( |N[S]| \le (1-e^{-1})t/2 \Big) \le \exp \big[- H(1/2) \cdot (1-e^{-1})t \big] \le e^{ - 0.09 t},
$$
where  $H(x) = x\ln x-x+1$ is again as in Lemma~\ref{lem:chernoff}. The probability that the necessary condition fails for at least one set $S$ with $s_0 < |S| \le \eps t$ is therefore at most
$$
\sum_{s=s_0+1}^{\eps t} {t \choose s} e^{- 0.09 t} \le e^{-0.01 t} = e^{-\Omega( tp )}.
$$

Finally, let $S \subseteq X$ with $\eps t < |S| = s \le t$. If $S$ fails the test, then there exists $T \subseteq Y$ of cardinality $t-s+1$ such that there is no edge between $S$ and $T$. Hence, the probability that the condition fails for at least one set $S$ with $\eps t < |S| \le t$ is at most
\begin{eqnarray*}
\sum_{\eps t < s \le t} {t \choose s} {t \choose t-s+1} (1-p)^{s(t-s+1)} 
&\le& 
\sum_{\eps t < s \le t} t^{t-s} \cdot t^{t-s+1} \cdot \exp\big[ -ps(t-s+1) \big] \\
& \le & 
\sum_{\eps t < s \le t} \exp\big[ (t-s+1)2\ln t - \eps p t (t-s+1)  \big] \\
& = & 
\sum_{\eps t < s \le t} \exp\big[ (t-s+1)(2\ln t - \eps p t)  \big] \\
&\le& \sum_{\eps t < s \le t} \exp \big[ - (t-s+1) \eps p t / 2 \big] \\
& = & O\big( e^{-\eps p t/2} \big),
\end{eqnarray*}
where we have used that we can choose $K$ large enough for $2\ln t < \eps K \ln t / 2$ to hold in the penultimate line.
We conclude that, using Hall's theorem:
\[ \begin{array}{rcl}
\Pee(\text{$\B(t,p)$ has a perfect matching}) 
& = & 
1 - O( t e^{-tp/100} ) - e^{-0.09 tp } - O( e^{-\eps pt / 2 } ) \\
& = &  
1 - O( t e^{-\Omega(tp)} ),
\end{array} \]

\noindent
as required.
\end{proof}

Before we can proceed with our upper bound on the acquaintance time of the 
percolated random geometric graph we need some more preparations.
A quite precise result on the acquaintance time of the binomial random graph $\G(n,p)$ was already given in~\cite{KMP}, but
we again require a version with precise quantitative bounds on the error-probabilities.
The following result will serve our purposes.

\begin{theorem}\label{lem:ER}
There exists a constant $\gamma > 0$ such that for all $k \leq t/1000$:

\[ \Pee\Big[ \AC( \G(t,p) ) \leq k \Big] \geq  1 - t^2 e^{-\gamma pk }, \]

\noindent
for all $t\in\eN$ and $0 < p < 1$.
\end{theorem}

\begin{proof}
In order to avoid technical problems with events not being independent, we use a classic technique known as \emph{multi-round exposure}. (In fact, we will use a three-round exposure here.) 
The observation is that a random graph $G \in \G(t,p)$ can be viewed as a union of three independently generated random 
graphs $G_1, G_2, G_3 \in G(t,\p)$, with $\p$ defined by: 
$$
p=1-(1-\p)^3. 
$$
(See, for example,~\cite{Bollobas2ndEd, randomgraphs} for more information).
Let us observe that $p \geq \p \geq p/3$.

Firstly, let us focus on $G_1=(V,E) \in \G(t,\p)$. Our goal is to show that, with probability at least 
$1 - O\left(t e^{-\Omega(tp)} \right)$, $G_1$ contains a path $P$ of length $0.9 t$. 
We consider the following process. Select any vertex $v_1 \in V$ and expose all edges in $G_1$ from $v_1$ to other vertices of $V$. If at least one edge is found, select any neighbour $v_2$ of $v_1$ and expose all edges from $v_2$ to $V \setminus \{v_1,v_2 \}$ with a hope that at least one edge is discovered and the process can be continued. The only reason for the process to terminate at a given round is when no edge is found. The probability that the process does \emph{not} stop before discovering a path $P$ of length $0.9 t$ is equal to
$$
\begin{array}{rcl}
\Pee( \text{$G_1$ contains a path of length $\geq 0.9 t$} ) 
& \geq &
(1-(1-\p)^{t-1}) 
\cdots (1-(1-\p)^{0.1t+1}) \\
& \ge & 
(1-e^{-0.1\p t})^{0.9 t} \\
& \ge &
1 - 0.9 t e^{-0.1\p t}.
\end{array} 
$$

\noindent
Now let $k \leq t/1000$ be arbitrary.
%
%
Conditioning on the event that $G_1$ has at least one path of length $0.9 t$, 
let us fix a path $P \subseteq G_1$ of length $V(P) \geq 0.8 t$ such that $V(P)$ is a multiple of $k$.
We now consider $G_2=(V,E) \in \G(t,\p)$. It follows from Lemma~\ref{lem:perfect_matching} that, with probability at least $1-O( t e^{\Omega(tp)} )$,  there is a matching $M \subseteq G_2$ between $V(P)$ and $V \setminus V(P)$ that saturates $V \setminus V(P)$. 
We call agents occupying $V(P)$ \emph{active} and agents occupying $V \setminus V(P)$ \emph{inactive}. 

 We split the path $P$ into many paths, each on exactly $k$ vertices. 
 This partition also divides the active agents into $O( t/k )$ teams, each team consisting of $k$ agents. 
 Every team performs (independently and simultaneously) the strategy from Lemma~\ref{lem:hamilton}. 
 This certainly results in every pair of active vertices on the same team getting acquainted.
%

Next, we will consider the probability that an active agent $x$ gets acquainted to an agent $y$ that is either on a different team or inactive.
It follows from Lemma~\ref{lem:hamilton} that agent $x$ visits $k$ distinct vertices.  
Since $y$ either belongs to a different team or is inactive, the pair $x,y$  occupy at least $k$ distinct pairs of vertices during the process. 
Considering only those edges in $G_3 \in \G(t,p)$, the probability that the two agents never got acquainted is at most
$$
\Pee( \text{$x,y$ do not get acquainted } ) \leq (1-\p)^{k} \leq e^{-\p k}.
$$
Since there are at most $t \choose 2$ pairs of agents, the union bound shows:
\[ \Pee( \text{ all active vertices get acquainted to each other and all inactive vertices} ) 
\geq 1 - O( t^2 e^{-\p k} ). \]

Finally, to also acquaint the inactive vertices with each other, we simply use the matching $M$ to place them
on $P$, and repeat the ``teams" strategy. This way they all pairs will indeed get acquainted.
Summarizing, we have
\begin{eqnarray*} 
\Pee(\text{all pairs get acquainted}) &\geq& 1 - 0.9 t e^{-0.1\p t} -O( t e^{-\Omega(tp)} ) - O( t^2 e^{-k\p} ) \\
&=& 1 - O\left( t^2 e^{-\Omega( k\p )} \right), 
\end{eqnarray*}

\noindent
which concludes the proof.
\end{proof}

Now, we are ready to come back to the upper bound for the acquaintance time of percolated random geometric graphs. 
We will prove the following upper bound

\begin{lemma}\label{lem:gnrub}
Let $\eps > 0$ be arbitrary.
There is a constant $K > 0$ such that if the sequences $(r_n)_n$ and $(p_n)_n$ are such that 
$p_n < 1-\eps$ and $p_n n r_n^2 \geq K \ln n$ for all $n$ then for $G \in \G(n,r_n,p_n)$ we have

\[ \AC(G) = O\left( \frac{\ln n}{r_n^2 \cdot p_n} \right) \quad \text{ w.h.p. }
\]
\end{lemma}

\begin{proof}
We want to mimic the strategy introduced for dense (classic) random geometric graph $\G(n,r_n)$---see subsection~\ref{sec:classic_dense}. 
Let us recall the setting briefly.  $\Dcal_n$ denotes the dissection of the unit square into $m_n^2$ ($m_n := \lceil 1000/r_n \rceil$) equal squares called cells.
For a given cell $c \in \Gamma_n$, $V(c)$ denotes the set of points of $\Xcal_n$ that fall in $c$;  $\mu_n := n / m_n^2$ is the expectation 
$\Ee |V(c)|$ and by Lemma~\ref{lem:wassenneus}, w.h.p, $0.9 \cdot \mu_n \leq |V(c) \leq 1.1 \cdot \mu_n$ for every cell.
Again each $c\in\Gamma_n$, $V(c)$ is partitioned into three parts
$V_1(c), V_2(c), V_3(c)$, each of cardinality at most $0.4 \cdot \mu_n$.
For each pair $1 \leq i < j \leq 3$ and each cell $c\in\Gamma_n$, $W_{ij}(c) \subseteq V(c)$ is a set of cardinality exactly
$$
t := \lfloor 0.9 \mu_n \rfloor, 
$$
such that $V_i(c) \cup V_j(c) \subseteq W_{ij}(c)$; and set $W_{ij} := \bigcup_{c\in\Gamma_n} W_{ij}(c)$.
$G_{ij} = G[W_{ij}]$ denotes the subgraph induced by $W_{ij}$.

Let $E_{ij}$ denote the event that for each adjacent pair of cells $c, d \in \Dcal_n$ there is a
perfect matching between $W_{ij}(c)$ and $W_{ij}(d)$.
It follows from Lemma~\ref{lem:perfect_matching} that

\[ \Pee( E_{ij} ) \geq 1 - 4 m_n^2 t e^{-\Omega( t p_n ) }
\geq 1 - 4 n^4 e^{ - \Omega( p_n n r_n^2 ) }
= 1 - o(1), \]

\noindent
where the last inequality holds since $p_n n r_n^2 \geq K \ln n$ with $K$ a sufficiently large constant.

Now let us set 
$$
k := \min\left( \frac{C \cdot \ln n}{p_n}, \frac{t}{1000} \right).
$$
with $C>0$ a constant to be chosen more precisely later, and let $F_{ij}$ denote the event that for every 
adjacent pair of cells $c, d \in \Dcal_n$ we have that 
$\AC( G[W_{ij}(c) \cup W_{ij}(d)] ) \leq k$.
It follows from Lemma~\ref{lem:ER} that
\[ \Pee( F_{ij} ) 
\geq 
1 - m_n^2 \cdot O\left( t^2 e^{-\Omega( k p_n )} \right)
= 
1 - O\left( n^4 e^{-\Omega( \min( C \ln n, p_n n r_n^2 )) } \right)
= 
1 - o(1),
\]

\noindent
using that $p_n n r_n^2 \geq K \ln n$ and that we can assume $C, K > 0$ are sufficiently large for
the last equality to hold.

We have seen that, w.h.p., $E_{ij}$ and $F_{ij}$ hold for each $1\leq i < j \leq 3$.
If $E_{ij}, F_{ij}$ both hold, we have the following strategy to acquaint all agents on $W_{ij}$.
Again we treat the set of agents initially on a cell $c \in\Dcal_n$ as a group.
We move these groups from cell to cell following the strategy of Theorem~\ref{thm:angel}, where 
for each move of the original strategy for the $m\times m$ grid, we use the perfect matchings
(that exist because $E_{ij}$ holds)
between adjacent cells to transfer entire groups.
After each move of the grid strategy, we do the following. Observe that the grid can be covered by four
matchings $M_1, \dots, M_4$ (take for instance horizontal/vertical edges with odd/even $x$/$y$-coordinates of the leftmost/top point).
For each such matching $M_i$ we do the following simultaneously for each of its edges $cd \in M_i$. 
We acquaint the groups of agents currently on $W_{ij}(c) \cup W_{ij}(d)$ and return them
to their starting points in $2 k$ moves (this can be done since $F_{ij}$ holds, and using Lemma~\ref{lem:return}).
Repeating this procedure for each of $W_{12}, W_{13}, W_{23}$, it is clear that this way we do acquaint all agents 
with each other in at most
\[ O\left( m_n^2 \cdot k \right) = O\left( \frac{\ln n}{r_n^2 \cdot p_n} \right), \]

\noindent
moves, as required.
\end{proof}

\section{Conclusion and further work}

In this article, we have determined the likely value of the acquaintance time of random geometric graphs up to the leading constant, whenever
the graph is w.h.p.~connected. 
A very natural question is thus to also find the leading constant (if it even exists).

\begin{problem}
Find a more detailed asymptotic description of the acquaintance time of random geometric graphs.
\end{problem}

For percolated random geometric graphs we have also found the likely value of the acquaintance time, but
we needed a slightly stronger assumption on the sequences $r_n, p_n$. Namely, these parameters needed to be chosen in such a 
way that all degrees are already slightly larger than $\sqrt{n}$. On the other hand we were able to provide an upper bound
that already works close to the connectivity threshold for percolated graphs.

\begin{problem}
Determine the likely value of the acquaintance time of percolated random geometric graphs
when $p_n n r_n^2 = \Omega( \ln n )$ and $p_n n r_n^2 = O( n^{1/2+o(1)} )$.
\end{problem}

\noindent
And, of course it would again be nice to have more detailed asymptotics.

\begin{problem}
Find a more detailed asymptotic description of the acquaintance time of percolated random geometric graphs.
\end{problem}

\bibliographystyle{plain}
\bibliography{ReferencesMullerPralat}

\appendix


\section{Proof (sketch) of Lemma~\ref{lem:structure}\label{sec:structure}}

We briefly sketch how the proof of Lemma 3.1 in~\cite{BDFMS} should be adapted to 
yield a proof of our Lemma~\ref{lem:structure}.
The main adaptation that is needed is that their Lemma 3.1 needs to be altered to allow for values of $T$ that are a
small constant times $\ln n$, as follows:

\begin{lemma}\label{lem:cells}
Let $\eta, K, \eps > 0$ be arbitrary but fixed, and let $m$ be given 
by~\eqref{eq:mdef} and let $\Xcal_n$ be as in~\eqref{eq:Xdef}.
There exists a $\delta = \delta(\eta, K, \eps) > 0$ such that for every $T \leq \delta \cdot \ln n$, w.h.p., 
the following hold:
\begin{enumerate}
\item\label{itm:cells.i} Out of every $K\times K$ block of cells, the area
of the bad cells inside the block is at most
$(1+\eps)\ln n/n$;
\item\label{itm:cells.ii} Out of every $K\times K$ block of cells
touching the boundary of the unit square,
the area of the bad cells inside the block is at
most $(1+\eps)\ln n / 2 n$.
\item\label{itm:cells.iii} Every $K\times K$ block of cells touching a corner
contains only good cells.
\end{enumerate}
\end{lemma}

The proof of Lemma 3.1 in~\cite{BDFMS} relies only on the Chernoff bound and it is
easily seen that an (almost) verbatim repeat of the proof will prove the above lemma.
The rest of the proof of our Lemma~\ref{lem:structure} can now follow the proof of Lemma 3.3 from~\cite{BDFMS} (almost) verbatim 
since the proof given there does not rely on any probabilistic calculations (only geometric ones), besides their Lemma 3.1.

\section{The proof of Lemma~\ref{lem:crucial}\label{sec:crucial}}

Before we can start describing the proof of Lemma~\ref{lem:crucial}, we need to cover
some more background and notation.
The usual random geometric graph $G(n,r) = G(\Xcal_n;r)$ is sometimes also called
the {\em binomial random geometric graph}.
It is often useful to switch to a ``Poissonized" version of the random geometric graph.
By this we mean the following. Let $N_n \isd \Po(n)$ be independent
of $X_1,X_2,\dots$, and set (following~\cite{PenroseBoek}):

\begin{equation}\label{eq:Pdef}
\Pcal_n := \{ X_1,\dots, X_{N_n}\}.
\end{equation}

\noindent
(Thus $\Pcal_n$ is a Poisson process with intensity $n$ on the unit square
and intensity 0 elsewhere---see for instance~\cite{Kingmanboek} for background on Poisson processes.)
The {\em Poisson random geometric graph} is defined as
$\GPo(n,r) := G(\Pcal_n; r)$.
We will make use of the following result from~\cite{BDFMS}, which is nearly identical to Theorem 1.6 of~\cite{PenroseBoek}.

\begin{theorem}[\cite{BDFMS}]\label{thm:palm}
Let $\Pcal_n$ be as in~\eqref{eq:Pdef}, and let $h(a_1,\dots,a_k;A)$ be a bounded
measurable function defined
on all tuples $(a_1,\dots,a_k;A)$ with
$A \subseteq \eR^2$ finite and $a_1,\dots,a_k \in A$. Let us write
\[
 Z := \sum_{a_1,\dots,a_k \in \Pcal_n, \atop a_1,\dots, a_k \text{ distinct }}
 h(a_1,\dots,a_k;\Pcal_n).
\]
Then
\[
\Ee Z
= n^k \cdot \Ee h( Y_1,\dots,Y_k; \{Y_1, \dots, Y_k\} \cup \Pcal_n ),
\]
where $Y_1,\dots, Y_k$ are i.i.d.~uniform on the unit square, and are independent of
$\Pcal_n$.
\end{theorem}

\noindent
We are now ready for the the proof of Lemma~\ref{lem:crucial}, which 
follows that of Lemma 3.5 in~\cite{BDFMS}.

\begin{proofof}{Lemma~\ref{lem:crucial}}
Let us first note that if we increase the parameter $r$ then 
some obstructions might cease to be obstructions, while vertices that 
were safe for the smaller $r$ will still be safe.
So in particular, it suffices to prove the statement for
$r_n := \sqrt{\ln n / \pi n}$.

Next, let us remark that, by Lemmas~\ref{lem:structure} and~\ref{lem:cells}, 
no point of any obstruction is within $100 r$ of a corner of the unit square.
We will use an appropriate first moment argument to bound the number of obstructions that are not close
to a corner and have too few crucial vertices.
For the moment, let us fix $a, c \leq K \ln n$ where $K := \lceil \pi / \eta^2 \rceil \cdot C$ with $C = C(\eta)$ as 
in Lemma~\ref{lem:maxcell}. (So in particular, w.h.p., the maximum degree is at most $K\ln n$. And hence, w.h.p., every obstruction 
has at most $K \ln n$ points as well.)
It is convenient to switch to the Poissonized version of the random graph for the moment, first show
there are no such obstructions in the Poissonized setting, and then deduce the result for the original 
setting
Let us say that $u,v$ is an {\em $(a,c)$-pair} (with respect to $r, V$) if:

\begin{itemize}
\item[\pr{1}] $\norm{u-v} < r/100$;
\item[\pr{2}] $B(u, r-\norm{u-v}) \setminus B(u, \norm{u-v} )$ contains exactly $a$
points of $V \setminus \{u,v\}$;
\item[\pr{3}] 
$B(u,\norm{u-v})$ contains exactly $c$ points of $V \setminus \{u,v\}$.
\end{itemize}

To motivate this definition, note that if $A$ is an obstruction, and we take 
$u,v \in A$ a pair realizing $\diam(A)$, then $(u,v)$ will be an $(a,c)$-pair with
$c = s-2$. If, furthermore~\str{1}--\str{5} hold, then each of the 
$a$ points in $B(u, r-\norm{u-v})$ must be crucial.
(Each such point must be safe, otherwise it would have been part of the obstruction.)

Let $\Rsde^{(a,c)}$ denote the number of $(a,c)$-pairs $(X_i,X_j)$ in $\Pcal_n$
for which $X_i$ is within $100r$ of the boundary of $[0,1]^2$, but not within $100 r$
of a corner; let $\Rmdl^{(a,c)}$ denote the number of $(a,c)$-pairs $(X_i,X_j)$
for which $X_i$ is more than $100 r$ away from the boundary of $[0,1]^2$, and set 
$R^{(a,c)} = \Rsde^{(a,c)} + \Rmdl^{(a,c)}$.

For $0 < z < r/100$ let us write:

\[ \begin{array}{l}
\mu_1(z) := n \cdot \area(B(u;r-z) \setminus B(u;z)), \\
\mu_3(z) := n \cdot \area( B(u;z) ),
\end{array} \]

\noindent
where $u,v \in\eR^2$ are two points with $\norm{u-v}=z$.

Let us now consider $\Rmdl$.  Using Theorem~\ref{thm:palm} we find

\begin{equation}\label{eq:RmdlFinal}
\begin{array}{rcl}
\Ee \Rmdl^{(a,c)}
& =  &
\displaystyle
n^2 \int_{[100r,1-100r]^2} \int_{B(v;r/100)}
\frac{\mu_1(\norm{u-v})^{a}e^{-\mu_1(\norm{u-v})}}{a!} \cdot 
\frac{\mu_3(\norm{u-v})^{c}e^{-\mu_3(\norm{u-v})}}{c!}
 {\dd}u{\dd}v \\
& =  &
\displaystyle
n^2 (1-200 r)^2 \int_0^{r/100}
\frac{\mu_1(z)^{a}e^{-\mu_1(z)}}{a!} \cdot
\frac{\mu_3(z)^{c}e^{-\mu_3(z)}}{c!}
 2\pi z{\dd}z \\
& \leq  &
\displaystyle n^2 \int_0^{r/100}
\left( \pi n r^2 \right)^{a}
\left( \pi n z^2 \right)^{c}
e^{ - \pi n r^2 - n r z } 2\pi z{\dd}z \\
& = &
\displaystyle
O\left(
n^{2+c} \cdot \ln^a n \int_0^{r/100}
e^{- \ln n - n r z }  z^{2c+1}{\dd}z
\right) \\
& = &
\displaystyle
O\left(
n^{1+c} \cdot \ln^a n \int_0^{r/100}
e^{- n r z }  z^{2c+1}{\dd}z
\right) \\
& = &
\displaystyle
O\left(
n^{1+c} \cdot \ln^a n \cdot (nr)^{-(2+2c)}
\right) \\
& = &
O\left(
\ln^a n \cdot (nr^2)^{-(1+c)}
\right) \\
& = &
\displaystyle
O\left(
\left(\ln n\right)^{a-(c+1)} \right).
\end{array}
\end{equation}

\noindent
Here we have used a switch to polar coordinates to get the second line;
and the change of variables $y = nrz$ to get the sixth line.

Now we turn attention to $\Rsde^{(a,c)}$.
By Theorem~\ref{thm:palm} and a switch to polar coordinates
and where $z=||u-v||$ and $w$ is the distance of the nearest
of $u,v$ to the boundary:
\[
\begin{array}{rcl}
\Ee \Rsde^{(a,b)}
& \leq  &
\displaystyle
800 n^2 r \int_0^{r/100}
\left( \pi n r^2 \right)^{a}
\left( \pi n z^2 \right)^{c}
e^{ - \frac{\pi}{2} n r^2 - \frac12 n r z} \pi z{\dd}z \\
& = &
\displaystyle
O\left(
n^{2+c} \cdot r \cdot \ln^a n \int_0^{r/100} z^{2c+1}
e^{ - \frac{\pi}{2} n r^2 - \frac12 n r z} {\dd}z
\right) \\
& = &
\displaystyle
O\left(
n^{\frac32+c} \cdot \ln^a n \int_0^{r/100} z^{2c+1}
e^{ - \frac12 n r z} {\dd}z
\right) \\
& = &
\displaystyle
O\left(
n^{\frac32+c} \cdot r \cdot \ln^a n \cdot (nr)^{-(2c+2)}
\right) \\
& = &
O\left(
\ln^a n \cdot (nr^2)^{-(\frac32+c)}
\right) \\
& = &
O\left(
\left(\ln n\right)^{a-\frac32-c}
\right).
\end{array}
\]

\noindent
We find that, for all $a \leq K\ln n, a+100 \leq c \leq K\ln n$:

\begin{equation}\label{eq:Rac}
\Pee( R^{(a,c)} > 0 ) = O( \ln^{-100} n ), 
\end{equation}

\noindent
This is of course still in the Poissonized setting.
Let $\tilde{R}^{(a,c)}$ be the corresponding number of pairs wrt.~$\Xcal_n, r$. (I.e.~the number of $(a,c)$-pairs in $\Xcal_n$ that are not close to a corner.)
Let us write $k := a+c+2$ for notational convenience.
If we condition on the event that $\Rtil^{(a,c)} > 0$ then we can sample a 
sequence $(Y_1,\dots, Y_{k}) \subseteq (\Xcal_n)^k$, uniformly at 
random from all such sequences where $(Y_1, Y_2)$ is an $(a,c)$-pair, 
$Y_3, \dots, Y_{c+2} \in B(Y_1, \norm{Y_1-Y_2})$ and
$Y_{c+3}, \dots, Y_k \in B(Y_1, r - \norm{Y_1-Y_2})$.
By symmetry considerations, every sequence without repetitions from $(\Xcal_n)^k$ will be equally likely
Let $E$ denote the event that $Y_1, \dots Y_k \in \{ X_i : i \leq n-n^{.99} \}$.
The symmetry considerations show that:

\[ \Pee\left( E^c | \Rtil^{(a,c)} > 0 \right) \leq k \cdot n^{-.01} = o(1). \] 

\noindent
(We condition on $\Rtil > 0$ so that $Y_1,Y_2, \dots, Y_{k}$ are defined.)
Let $F$ denote the event that $n-n^{.99} < N_n \leq n$.
Since $N_n$ is independent of the sequence of points $X_1, X_2, \dots$, the event $F$ is independent of the events
$E$ and $\{ R' > 0 \}$.
Hence we also have

\begin{equation}\label{eq:E12} 
\Pee( E^c | \Rtil^{(a,c)} > 0,  F ) = o(1). 
\end{equation}

Next, let us observe that if $E$ and $F$ hold then we must have that $R^{(a,c)} > 0$.
(The points $Y_1, Y_2, \dots, Y_k$ will all be part of $\Pcal_n$ as well as $\Pcal_n$ and 
since $\Pcal_n$ is contained in $\Xcal_n$ there are no additional points that 
could mess up the values of $a$ or $c$ by falling in the wrong region.)
It follows that

\begin{equation}\label{eq:aapnoot} 
\begin{array}{rcl}
\Pee( R^{(a,c)} > 0 ) 
& \geq & 
\Pee( E | \Rtil^{(a,c)} > 0, F ) \cdot
\Pee( \Rtil^{(a,c)}  > 0, F ) \\
& = & 
\Pee( E | \Rtil^{(a,c)} > 0, F ) \cdot \Pee( \Rtil^{(a,c)} > 0 ) \cdot \Pee( F ) \\
& = & 
(1 - o(1) ) \cdot \Pee( \Rtil^{(a,c)} > 0 )  \cdot \left(\frac12 -o(1) \right), 
\end{array} 
\end{equation}

\noindent
using~\eqref{eq:E12} and the Chernoff bound
for the last line.
Combining~\eqref{eq:Rac} and~\eqref{eq:aapnoot}, it follows that we must have

\[ \Pee( \Rtil^{(a,c)}  > 0 ) = O( \ln^{-100} n ), \]

\noindent
as well.
Setting $\Rtil := \displaystyle \sum_{a \leq K \ln n, \atop a + 100 \leq c \leq K \ln n} \Rtil^{(a,c)}$, we see that we
have $\Rtil = o(1)$ w.h.p.
This finishes the proof that (w.h.p.)\ there
cannot be any $s$-obstruction with less than $s-100$ crucial vertices.
\end{proofof}

\end{document}